\documentclass[11pt,leqno]{article}

\usepackage{amssymb,amsmath,amsthm, url,rotating}
 \usepackage{graphicx,epsfig}
 \usepackage{xcolor,graphicx}
  \usepackage{hyperref}
\newcommand{\n}{\noindent}

\newcommand{\vp}{\varepsilon}
\newcommand{\bb}[1]{\mathbb{#1}}
\newcommand{\cl}[1]{\mathcal{#1}}

\newcommand{\ovl}{\overline}

\theoremstyle{plain}
\newtheorem{thm}{Theorem}[section]
\newtheorem{lem}[thm]{Lemma}

\theoremstyle{definition}

\theoremstyle{remark}
\newtheorem{rem}[thm]{Remark}

\numberwithin{equation}{section}

\def\R{\bb R}

\def\C{\bb C}

\def\P{\bb P}
\def\T{\bb T}

\def\d{\delta}

\def\N{\bb N}

\def\P{\bb P}
\def\T{\bb T}

\def\d{\delta}

\begin{document}

\title{On the metric entropy of the Banach-Mazur compactum}

\author{by\\
Gilles Pisier\footnote{Partially supported   by ANR-2011-BS01-008-01.}\\
Texas A\&M University\\
College Station, TX 77843, U. S. A.\\
and\\
Universit\'e Paris VI\\
Inst. Math. Jussieu, \'Equipe d'Analyse Fonctionnelle, Case 186, \\
75252
Paris Cedex 05, France}

 \maketitle

\begin{abstract} We    study  the metric entropy of the
metric space $\cl B_n$ of all $n$-dimensional Banach spaces (the so-called 
Banach-Mazur compactum) equipped with the Banach-Mazur (multiplicative) ``distance" $d$.
We are interested either   in estimates independent of the dimension
or in asymptotic estimates when the dimension tends to $\infty$.
For instance, we prove that,
if $N({\cl B_n},d, 1+\vp)$ is the smallest number of ``balls" of ``radius" $1+\vp$
that cover $\cl B_n$, 
 then
   for any $\vp>0$ we have
 $$0<\liminf_{n\to \infty}n^{-1} \log\log N(\cl B_n,d,1+\vp)\le 
\limsup_{n\to \infty}n^{-1}  \log\log  N(\cl B_n,d,1+\vp)<\infty.$$
We also prove an analogous result
for the metric entropy of the set of $n$-dimensional operator spaces
equipped with the distance $d_N$ naturally associated to $N\times N$-matrices with operator entries.
In that case $N$ is arbitrary but our estimates are valid independently of $N$.
In the Banach space case (i.e. $N=1$) the above upper bound is part of the folklore, and the lower bound is at least partially known
(but apparently has not appeared in print). 
While we follow the same approach in both cases,
the matricial case requires more delicate ingredients, namely
estimates (from our previous work) on certain $n$-tuples of $N\times N$ unitary matrices known as ``quantum expanders".\end{abstract}

Let $\cl B_n$ denote the set of all   $n$-dimensional Banach spaces.
Actually we wish to  identify two
  spaces   if they are isometrically isomorphic. Thus, although it would be more proper
  (but heavier) to describe $\cl B_n$ as the set of equivalence classes
  for this equivalence relation, we adopt the common abuse  to identify a space $E$ 
  with its equivalence class.
We equip this space with the (multiplicative) metric  defined for any pair $E,F\in \cl B_n$  by
$$d (E,F)=\inf \{ \|u\|  \|u^{-1}\|  \}$$
where the inf runs over all the  isomorphisms $u:\ E \to F$.
For any   $G\in \cl B_n$ we have
$d (E,F)\le d (E,G)d (G,F)$, so that $(E,F)\mapsto \log d(E,F)$ is indeed a distance.
It is easy to see (by the compactness of the unit ball of $B(E,F)$) that
$E,F$ are isometric iff $d(E,F)=1$, and in that case we declare that $E=F$. Therefore $\cl B_n$  equipped with $\log d$
is a bona fide metric space. It is a classical fact (presumably due to Banach and Mazur) that this metric space is  compact, and hence is often called the Banach-Mazur compactum.
Since it is more convenient to work with $d$ than with $\log d$, we
will abusively define an open ball of radius $r$ in $(\cl B_n,d)$
to be any set of the form $\{E\mid d(E,E_0)<r\}$, and we call it the ball of radius $r$ centered at $E_0$.
Note that this is non void only if $r>1$.

In this note we wish to tackle the rather natural problem of estimating the metric entropy
$H_n(\vp)= \log_2 N(\cl B_n,d,1+\vp)$
of this compactum, or equivalently its covering number
$N(\cl B_n,d,1+\vp)$. By $N(\cl B_n,d,1+\vp)$ we mean
  the smallest number of open balls
of radius $1+\vp$
 that cover $\cl B_n$. 
By compactness we of course know that
$N(\cl B_n,d,1+\vp)<\infty$ for any $\vp>0$.\\
The metric entropy $H_n(\vp)$ was already studied in \cite{Br} (see also the more recent survey \cite{Br2}):
Indeed, it is proved in \cite{Br} that $H_n(\vp) \approx \vp^ {{\frac{1-n}{2}} }$. Note however,
that this equivalence is for each fixed dimension $n$ and for sufficiently small $\vp>0$
(the range of which may depend on $n$). In sharp contrast we are interested
in estimates independent of the dimension, 
or with asymptotic estimates when $n\to \infty$ with a view to infinite dimensional applications.
Our interest lies in estimates
of  $N(\cl B_n,d,1+\vp)$ (or possibly of $N(\cl B_n,d,1+\vp_n)$) when
$n\to \infty$ and $\vp>0$ is either fixed or depends on $n$ in a prescribed way. 

Our proofs \S \ref{s1} are all rather simple and direct, and the results are partially known, but mostly unpublished
(see however Remarks \ref{sz} and  \ref{lrt} below). One of our goals
is to stimulate further research on this theme.

 We prove in Theorem \ref{t1}
 that for any $\vp>0$ we have
 \begin{equation}\label{eq01}0<\liminf_{n\to \infty} n^{-1} \log\log N(\cl B_n,d,1+\vp)\le 
\limsup_{n\to \infty}n^{-1}  \log\log  N(\cl B_n,d,1+\vp)<\infty.\end{equation}

Moreover, it turns out that for any $R>1+\vp$  and any choice of $E\in \cl B_n$, 
if we replace the whole space $\cl B_n$ by the single ball  $B_{E,R}\subset \cl B_n$,
with  center  $E$ and radius $R$, 
then we still have
  $$0<\liminf_{n\to \infty} n^{-1} \log\log N(B_{E,R},d,1+\vp).$$
This refinement is treated in \S \ref{s2-}.

We should recall that by Fritz John's famous theorem (see e.g. \cite{TJ,P1}), we have
$$d(E,\ell_2^n)\le \sqrt n$$ and hence
$$\forall E,F \in \cl B_n\quad d(E,F)\le   n.$$
Thus the Euclidean space $\ell_2^n$ is a natural center for $\cl B_n$ and
the $d$-diameter of $\cl B_n$ is at most $n$. Moreover, 
there is
a number $\delta>0$ such that for each $n$ there are $n$-dimensional spaces $E_n,F_n$
such that $d(E_nF_n)>\delta n$. This is due to Gluskin \cite{G} (see also \cite{Sz}).
It would be interesting to estimate the growth of  $N(\cl B_n,d, \vp_n)$
when $\vp_n$ is of the same order as $n$.
Similar questions can be raised for the ``packing number"
$M(\cl B_n,d, 1+\vp)$ which is defined as the maximal number of
elements in $\cl B_n$ at mutual distance $\ge 1+\vp$.
Note that $N(\cl B_n,d, 1+\vp )\le M(\cl B_n,d, 1+\vp)\le N(\cl B_n,d, (1+\vp)^{1/2})$
for any $\vp>0$.
 See Remarks \ref{glu} and \ref{lrt} below for some results in that direction,
 based on Gluskin's ``random Banach spaces".

We should mention that somewhat related results appear in \cite{STJ1,STJ2}.
Also \cite{G2,Sz3,Sz2,Sz4,BV} is  recommended reading to anyone interested   
in the subject of this note.

While  \S \ref{s1} is devoted to the metric entropy of the space of $n$-dimensional Banach spaces,
\S \ref{s2} treats the case of operator spaces, which can be viewed as ``non-commutative"  or
``matricial" analogues of Banach spaces. The latter are just Banach spaces
given together with a specific embedding in the algebra $B(H)$ of all bounded operators on a Hilbert space $H$.
There the Banach-Mazur distance $d$ is replaced by
the matricial analogue $d_N$ defined in \eqref{eq77} below, depending on the size $N$ of the matrices being used.
When $N=1$ we recover the distance $d$.

For the operator space case in \S \ref{s2}, the general structure of our argument
is modeled on that of \S \ref{s1}. However,
the ingredients   are more involved.  We make crucial use of our previous results from \cite{P9} on quantum expanders,
and we carefully explain in \S \ref{s3} why they are needed in the matricial case.

\begin{rem}\label{sz} 
I am grateful to S. Szarek for the following information. The upper bound in \eqref{eq01} is known and should be considered part of the folklore. The lower bound was known to Joram Lindenstrauss and Szarek
for say $\vp=1$ at least in the 1990's. The referee of \cite {STJ1}  insisted 
that it should be attributed to \cite{Br}, which, in retrospect,  was not appropriate
since the entropy estimates in \cite{Br}  were  for a fixed dimension
and for $\vp$ small enough. In any case, both upper and lower bounds in \eqref{eq01} did not seem to have appeared explicitly in print yet when we submitted this paper. See however Remark \ref{lrt} for an update concerning \cite{LRT}.\end{rem}

\section{Packing and covering the Banach-Mazur compactum}\label{s1}

In this section,  all the metric properties (such as nets and balls) are relative to the Banach-Mazur distance $d$.
Thus to abbreviate, we will denote simply
$N(\cl B_n,1+\vp)$ instead of $N(\cl B_n,d,1+\vp)$. \\ More generally, for any subset $S\subset \cl B_n$ 
and $r>1$ we denote by $N(S,r)$
 the smallest number of open balls (centered in $S$)
of radius $r$
 that cover $S$.  Note that $N(S,r)\le N(\cl B_n,r^{1/2})$.
 
 Similarly we will denote by $M(S,r)$ (``packing number") the maximal cardinality
 of a finite subset $A\subset S$
 such that $d(s,t)\ge r$ for all $s\not=t\in A$.   It is well known    (and easy to check) that
 for any $r\ge 1$
$N(S,r)\le  M(S,r)\le N(S,r^{1/2})$.

Our main result is:

\begin{thm} \label{t1} 
For any $r>1$,
there are   positive constants $b_r,c_r$
such that,    for any integer $n$  assumed large enough
(more precisely for $n\ge n_0(r)$),  we have
$$ \exp(\exp b_r n) \le N(\cl B_n,r)\le \exp(\exp c_r n).$$
\end{thm}

While the upper bound is easy and already known to specialists, the lower bound is a bit more delicate. Curiously, for the latter
we were inspired by the operator space analogues developed   in \cite{ P9}  following the ideas in
  \cite{JP}.
In the second part of this note, we expand on the first attempt
of \cite{P9} and give the matricial analogue of the preceding theorem.

 The first two  Lemmas are classical facts.
 \begin{lem}\label{l22}  Let $X$ be an $m$-dimensional Banach space.
Let $S(n,X)\subset \cl B_n$ be the subset formed by all the $n$-dimensional subspaces of $X$. 
Consider $0<{\xi}<1/n$ and let $R=(1+{\xi} n)(1-{\xi} n)^{-1}$.
Then $S(n,X)$ admits an $R$-net with at most $(1+2/{\xi})^{nm}$.
In other words,
$$N(S(n,X), R)\le (1+2/{\xi})^{nm}.$$
\end{lem}
\begin{proof}  By Auerbach's classical Lemma, any $E\in \cl B_n$ admits
a basis $e_j$ such that for any $n$-tuple of scalars $x=(x_j)$ we have
$$\sup|x_j|\le \|  \sum x_j e_j \| \le \sum |x_j|.$$
Let $T\subset B_X^m$ be a ${\xi}$-net in the unit ball of $X$
with  $|T|\le (1+2/{\xi})^{m}$ (see e.g. \cite[p. 49]{P1}).
Given an arbitrary $E\in S(n,X)$ with Auerbach basis $(e_j)$ we may choose
$f=(f_j)\in T^n$ such that $\sup\|e_j-f_j\|\le {\xi}$.
Then we have for any $n$-tuple of scalars $x=(x_j)$
$$\|\sum x_j e_j\|\le \|\sum x_j f_j\|+ {\xi} \sum |x_j|\le  \|\sum x_j f_j\|+ {\xi} n \sup |x_j|\le \|\sum x_j f_j\|+ {\xi} n \|\sum x_j e_j\|$$

$$\|\sum x_j e_j\|\ge \|\sum x_j f_j\|- {\xi} \sum |x_j|\ge  \|\sum x_j f_j\|- {\xi} n \sup |x_j|\ge \|\sum x_j f_j\|- {\xi} n \|\sum x_j e_j\|$$
and hence
$$  (1-{\xi} n) \|\sum x_j e_j\|   \le   \|\sum x_j f_j\|  \le (1+{\xi} n) \|\sum x_j e_j\|$$
from which we deduce that if we set $E_f={\rm span}[f_j]$ we have
$$d(E,E_f)\le R.$$
Thus $\{E_f\mid f\in T^n\}$ is the desired $R$-net, and $|T^n|\le |T|^n \le (1+2/{\xi})^{nm}$.
\end{proof}
\begin{lem}  Fix $0<\delta<1$. Let $m$ be the largest integer such that $m\le (1+2/\delta)^n$.\\
For any $n$-dimensional space $E$  there is a subspace $F\subset \ell_\infty^m$ such that 
$$d(E,F)\le  (1-\delta)^{-1} .$$
\end{lem}
\begin{proof}  Let $T$ be a $\delta$-net in the unit ball of $E^*$
with  $|T|\le (1+2/\delta)^{m}$ (see e.g. \cite[p. 49]{P1}).
We claim that for any $x\in E$ we have
$$(1-\delta) \|x\| \le \sup_{t\in T}|t(x)| \le \|x\| .$$
Indeed, choose $s\in B_{E^*}$ such that $\|x\|=|s(x)|$ and then $t\in T$ such that $\|s-t\|\le \delta$.
We have then
$$\|x\|=|s(x)|\le  |t(x)|+|(s-t)(x)|\le \sup_{t\in T}|t(x)|+ \delta \|x\|$$
from which the claim follows. The Lemma is then clear since
this claim implies $d(E,F)\le  (1-\delta)^{-1}$
where $F=\{\hat x\mid x\in E\}\subset \ell_\infty(T)$
 is the subspace defined by setting $\hat x(t)=t(x)$ for all $t\in T$.
  \end{proof}
  \begin{rem} See \cite[Cor. 1.2]{Bar} for a recent refinement of the estimate in Lemma \ref{l22}.
  \end{rem}  
  \begin{proof}[Proof of Theorem \ref{t1} (Upper bound)]  Fix $0<\vp<1$. Let $\xi= \vp/n$.
  Then $R=(1+\vp)(1-\vp)^{-1}$
   Let $m=[ (1+2/\delta)^n  ]$.  Combining the two Lemmas
   we find  
   $$N({\cl B}_n ,(1-\delta)^{-1}R)   \le  (1+2/{\xi})^{nm}\le (1+2n/{\vp})^{nm}\le (3n/{\vp})^{nm}=\exp(nm \log(3n/{\vp})$$
   and a fortiori
   $$N({\cl B}_n,  (1-\delta)^{-1}R)   \le\exp(n \log(3n/{\vp})\exp n \log(3/\delta)).$$
   So if we take, say $\delta=\vp$, there is clearly $c'_\vp>0$ such that
   $\exp(n \log(3n/{\vp})\exp n \log(3/\vp))\le \exp \exp c_\vp n$ for all $n$,
   and $(1-\delta)^{-1}R=(1+\vp)(1-\vp)^{-2}$, so we conclude
   $$N({\cl B}_n,  (1+\vp)(1-\vp)^{-2})   \le  \exp \exp c'_\vp n,$$
   which is clearly equivalent to an upper bound of the announced form, at least for all 
   $r= 1+\vp$ with $\vp$
   small enough. But, since the upper bound becomes easier for larger values
   of $\vp$, the    proof of the upper bound is complete.
\end{proof}
\def\t {\cl T}
\def\d{\delta}

\begin{lem}\label{l11}  Let $a>0$. Let $A\subset \{-1,1\}^n$ be a (measurable) subset with $\bb P(A)> a$.
  Then, for any $0<{\theta}<1$,  
 $A$ contains  a finite  subset $T\subset A$ 
 with $$|T|\ge (a/2)\exp {\theta}^2n/2,$$
  such that
 $$\forall s\not= t\in T\quad  |\sum s_j t_j| \le   {\theta} n.$$
 \end{lem}
 
\begin{proof} Let $\Omega=\{-1,1\}^n$ equipped with the uniform probability $\P$.
Let $T$ be a maximal subset of $A$ such that
 $\forall s\not= t\in T\quad  |\sum s_j t_j| \le   {\theta} n.$ Then, by the maximality of $T$,
 for all ${\omega}\in A$ there is $t\in T$ such that 
 $|\sum {\omega}_j t_j|>{\theta} n$. Therefore
 $\P(A)\le \sum_{t\in T} \P\{  |\sum {\omega}_j t_j|>{\theta} n\}\le |T| \P\{|\sum {\omega}_j |>{\theta} n\}$.
 The last inequality holds because, by translation invariance on $\{-1,1\}^n$, $\P\{\omega\mid  |\sum {\omega}_j t_j|>{\theta} n\}$
 is independent of $t=(t_j)$. But now it is a classical fact that
 $\P\{|\sum {\omega}_j |>{\theta} n\}\le 2 \exp -{\theta}^2n/2$ for any ${\theta}>0$
 and hence
 $$|T|\ge (a/2)\exp {\theta}^2n/2.$$
 From this follows that for all $n$ large enough (i.e. $n\ge n(a,{\theta})$)
 we have, say,  $$|T|\ge  \exp {\theta}^2n/3.$$
\end{proof}
\def\b{\beta}
 
 \begin{lem}\label{l3}  Let $\t $ be any finite set with even cardinality $K$.
Let $P(\t)$ denote the set of all the $2^{K}$  subsets of $\t$.
Let $\cl A\subset P(\t)$ be the subset formed of subsets
with cardinality $K/2$.
Then
$$\forall x\not =y\in \cl A\quad x\not \subset y \quad {\rm and}\quad  y\not\subset x, $$
and we have
$| \cl A|\approx \sqrt{2/\pi}  K^{-1/2} 2^K$. A fortiori,
we have for all $K$ large enough
$$| \cl A| \ge \exp K/2.$$
\end{lem}
\begin{proof}  Entirely elementary. 
\end{proof}

\begin{proof}[Proof of Theorem \ref{t1} (Lower bound)] 
Let $\t$  be the set $T$ appearing  in Lemma \ref{l11} 
for $A=\{-1,1\}^n$.  We may clearly assume $K$ even
(indeed removing one point will not spoil the estimates appearing below).
 Let $\cl A$ be as in Lemma \ref{l3}.
Recall that for any $x\in \cl A$ we have $x\subset \t \subset \{-1,1\}^n$. 
Let us denote by $\vp_j$ the coordinates on $\{-1,1\}^n$,
and by ${\vp _j}_{|x}\in \ell_\infty(x)$ the restriction of $\vp_j$ to the subset $x$.\\
We denote by $(e_j)$ the canonical basis of $\R^n$ (or  $\C^n$ depending on the context).\\
We  use the same notation for  the basis of $\ell_p^n$ for $1\le p\le \infty$.\\
We then define for any $x\in \cl A$
$$f^x_j=e_j\oplus {\vp _j}_{|x}  \in   \ell^n_\infty\oplus_{\infty} \ell_\infty(x) .$$
and 
$$E_x={\rm span}[f^x_j, 1\le j\le n]\subset  \ell^n_\infty\oplus_{\infty} \ell_\infty(x) .$$
Note $\dim(E_x)=n$.  We will take the simplified viewpoint that
  each space $E_x$ is equal to $\R^n $ (or $\C^n $)  equipped with the norm
  $$\forall a\in \R^n\quad \|\sum a_j  e_j\|=  \| \sum a_j f^x_j\|_{E_x}=\max\{ \sup |a_j|, \sup_{t\in x} |\sum a_j t_j|\}.$$
  Thus any linear map $u:\ E_y\to E_z$ can be identified with an $n\times n$-matrix.
Since $\|f_j^x\|=1$  we have for any     any $x\in \cl A$
\begin{equation}\label{eq12} \forall (a_j)\in \C^n\quad  \sup |a_j| \le  \| \sum a_j f^x_j\|\le  \sum |a_j| .\end{equation}

Then,
assuming $n$ large enough (so that ${\theta} n\ge 1$), for any $y\not=z\in \cl A$ we can find $s\in y\setminus z$  
such that 
$$\|\sum s_j f^y_j\|=n \quad {\rm but}\quad \|\sum s_j f^z_j\|\le{\theta} n.  $$
This proves  a lower bound $\|u\|\|u^{-1}\|\ge 1/{\theta}$ (and actually $\ge 1/{\theta}^2$) when $u:\ E_y\to E_z$ is the map
corresponding to the identity matrix. To obtain a similar lower bound for {\it any} $u$
the rough idea is that there are ``many more" $x$'s in $\cl A$ than there are linear maps
on an $n$-dimensional space.  \\
We will show that for any $r$ with   $1<r<1/{\theta}$ there is a subset $\cl A'\subset \cl A$ 
with $|\cl A'|\ge \exp\exp c_r n$ such that
$d(E_x,E_y)>r$ for any $x\not= y\in \cl A'$. (Perhaps one can show this even for $1<r<1/{\theta}^2$ but
we do not see how.)

 Let   $1+{\eta} = 1/(r{\theta} )$.    We will prove the following claim  
 (assuming we work with real spaces, but the complex case 
 is similar, requiring just $2n^2$ instead of $n^2$).\\
 {\bf Claim:} For any  fixed $x\in \cl A$ we have
    $$|\{y\in \cl A\mid d(E_x,E_y)<r\}|\le (1+4n/{\eta})^{n^2}.$$
 To prove this claim we use the classical fact that   the unit ball of the space $B(\ell_1^n, E_x    )$
 (being $n^2$ dimensional)  contains a ${\eta}/2n$-net of cardinality at most $(1+4n/{\eta})^{n^2}$.
  
 Note that,  by \eqref{eq12}, for any $y$, we have for any $n\times n$-matrix $u$
 \begin{equation}\label{eq330}\|u:\ \ell_1^n \to E_x\|\le  \|u:\ E_y \to E_x\|\le  n\|u:\ \ell_1^n \to E_x\|.\end{equation}
 
 For any $y$ with $d(E_x,E_y)<r$ there is a matrix $u^y$
 such that 
  \begin{equation}\label{eq331}\|u^y:\ E_y \to E_x\|\le 1 \quad{\rm and}\quad   \|(u^y)^{-1}:\ E_x \to E_y\|  < r.\end{equation}  Assume by contradiction
 that  $|\{y\in \cl A\mid d(E_x,E_y)<r\}|>(1+4n/{\eta})^{n^2}$. Then
 there must exist $y\not =z$ both in the set $|\{y\in \cl A\mid d(E_x,E_y)<r\}$
 such that the associated $ u^y$ and $ u^z$ are ${\eta}/2n$-close to the same point
 of the ${\eta}/2n$-net, and hence such that
 $\|    u^y- u^z :\ \ell_1^n\to E_x\| \le {\eta}/n$. 
 Therefore, by \eqref{eq330}, we have 
  $\|u^y-u^z:\ E_z\to E_x\|\le  {\eta},$  and hence by \eqref{eq331}
  \begin{equation}\label{eq332}\|u^y :\ E_z\to E_x\|\le 1+  {\eta}.\end{equation}
 But now recall that since $y\not =z$ with $y,z\in \cl A$, we already know that 
 $1/{\theta}\le \|Id:\ E_z\to E_y\|$. Thus we have
 $$  1/{\theta}\le \|Id:\ E_z\to E_y\| =\| (u^y)^{-1}u^y  :\ E_z\to E_y\|
 \le  \| (u^y)^{-1}:\ E_x\to E_y\|  \| u^y :\ E_z\to E_x\|$$
  and hence by \eqref{eq331} and \eqref{eq332}
 $$  1/{\theta}< r(1+{\eta}).$$
 Since this contradicts our initial choice of $r$, this proves the claim.\\
 Let $\cl X\subset \cl A$ be a maximal subset such that
 $d(E_x,E_y)\ge r$ for any $x\not= y\in \cl X$.
 Then $\cl A$ is covered by the balls of radius $r$ centered in the points of $\cl X$, and hence by the claim
 $$ |\cl A|\le |\cl X| (1+4n/\eta)^{n^2}\le |\cl X|  \exp{4n^3/{\eta} } .$$
 Thus we conclude
 $$|\cl X| \ge  |\cl A| \exp{-4n^3/{\eta} }.$$
 We now recall that we take $|A|=2^n$, $a=1$, $|T|=K\ge (1/2)\exp{ {\theta}^2 n/2}$
and  assuming $n$ large enough  we have
 $|\cl A|\ge \exp{K/2} =\exp( (1/4)\exp{ {\theta}^2 n/2})$.
 Therefore, elementary verifications lead to
 \begin{equation}\label{eq33} |\cl X| \ge  \exp [  \exp{ ({\theta}^2 n/4)}] .\end{equation}
 for all $n$ large enough (say $n\ge n_0({\theta},r)$).
At this point the lower bound follows immediately since $N(\cl X, r^{1/2})\ge M(\cl X, r)  = |\cl X|$.
\end{proof}
\begin{rem} One can replace $e_j$ by a suitable system of unimodular functions $w_j$ (such as the Walsh system, assuming $n$ dyadic)
for which it is known that $\|\sum \vp_jw_j\|_\infty\in o(n)$ on a set of large probability of choices of signs $(\vp_j)$. 
The proof then leads to a similar family of spaces $(E_x)$ but spanned by unimodular functions
in $\ell_\infty^m$ for   $m=n+K/2$.
\end{rem}
\begin{rem} The preceding argument actually shows more precisely:\\
For all $r>1$
$$\liminf_{n\to \infty} \frac{\log\log N( \cl B_n, r)}{n}\ge \liminf_{n\to \infty} \frac{\log\log M( \cl B_n, r^2)}{n}\ge \frac{1}{2r^2}
.$$
\end{rem}
\begin{rem}
Note that the
Elton-Pajor theorem from \cite{E,Pa} applies here to any of the spaces $E_x$.
Indeed, it is easy to check that the
average over all $\pm$ signs
of $\|\sum \pm f_j^x\| $ is $\ge c_{\theta} n$ for some $c_\theta>0$. This shows
that all these spaces (uniformly) contain  $\ell_1^k$'s
with dimension $k$ proportional to $n$.
\end{rem}

\begin{rem} The preceding construction of the spaces $E_x$ can be done
using a subset $T$
of the unit sphere  of $\ell_2^n$, denoted by $S(\ell_2^n)$. Indeed, assume that
we have
$$\forall s\not=t\in T\quad  |\langle s.t\rangle|\le {\theta}.$$
We set $K=|T|$ and (assuming $K$ even)  consider again the subset $\cl A\subset P(T)$
of subsets $x\subset T\subset S(\ell_2^n)$ with $|x|=K/2$.
Then for any $x\in \cl A$ we set
 $$e^x_j={\theta} e_j\oplus \vp^x_j  \in   \ell^n_2\oplus_{\infty} \ell_\infty(x) ,$$
 and $\cl E_x={\rm span}[e^x_j]\subset \ell_2^n\oplus_{\infty} \ell_\infty(x).$
 A similar reasoning leads to a separated family in $\cl B_n$.
 More precisely, if we denote by $K_{\theta}$ the smallest even $K$ for which there exists
 such a set $T$,
 and if we take, say, $\delta={\theta}^2$
 then we find for any $0<{\theta}<1$
 $$M( \cl B_n,  {\theta}^{-1}(1+{\theta})^{-1} )\ge    \gamma 2^{K_{\theta}} K_{\theta}^{-1/2}\exp{-4n^2/{\theta}^2 },$$
 which is significant if ${K_{\theta}}$ is significantly larger than $n^2/{\theta}^2$.
 See \cite[Chapter 9]{CS} for an account of the  works of Kabatiansky and Levenshtein estimating $K_{\theta}$.
\end{rem}

\begin{rem}\label{glu} Given a classical metric space $(T,d)$, suppose given a probability measure $\nu$ on $T$
such that any open ball of radius $r$ has $\nu$-measure $\le f(r)$ (resp. $\ge g(r)$). Then obviously we have
$M(T,d,r)\ge N(T,d,r)\ge 1/f(r)$ (resp. $N(T,d,r)\le M(T,d,r)\le 1/g(r/2)$). Gluskin's method in \cite{G,G2} uses  elements chosen at random in $\cl B_n$
according to a certain probability measure $\nu_n$ for which he proves
that (here $d$ denotes the Banach-Mazur  ``multiplicative" distance)
for some numerical constant $0<c<1$,
for any $E$ in the support of $\nu_n$, we have  
$$\nu_n\{ F\in \cl B_n\mid d(E,F) < c n \}  \le 2^{-n^2}.$$
See \cite[Prop. 1]{G2}. Thus it follows that,
if we denote by $\cl S_n$ the support of $\nu_n$, we have $$M(\cl B_n,  c n) \ge M(\cl S_n,  c n)\ge N(\cl S_n, c n)\ge 2^{n^2},$$
and a fortiori $N(\cl B_n, \sqrt{c n})\ge 2^{n^2}$.
\end{rem}
\begin{rem}\label{lrt} After a first version of this paper had been submitted, the paper \cite{LRT}
by Litvak, Rudelson and Tomczak-Jaegermann was brought to our attention.
It turns out that their  \cite[Corollary 2.4]{LRT} implies our Theorem \ref{t1}   (lower bound).
Their method is a variation of the Gluskin method. For any  $M$ such that $2n\le M\le e^n$ they produce a probability $\nu_{n,M}$
on ${\cl B}_n$ such that 
\begin{equation}\label{eq99}\nu_{n,M}\times \nu_{n,M} \left(\{  (E,F)\in    {\cl B}_n \times {\cl B}_n\mid d(E,F)< C{n/\log(M/n)}\}\right)\le  2 e^{-nM}\end{equation}
and consequently for any $E_0\in {\cl B}_n$ 
$$\nu_{n,M} (\{E\mid d(E,E_0)< \sqrt{Cn/\log(M/n)}\}  )\le \sqrt {2} e^{-nM/2}$$
where $C>0$ is an absolute constant.
This implies obviously
$$1=\nu_{n,M} ({\cl B}_n)\le N( {\cl B}_n,d, \sqrt{Cn/\log(M/n)}) \sqrt {2} e^{-nM/2},$$
so that choosing $M\approx n e^{nC/r^2}$  ($r>1$) one obtains the lower bound of our Theorem \ref{t1} in the form
$$     \exp{\exp{C'n/r^2} } \le N( {\cl B}_n,d, r)$$
where $C'>0$ is an absolute constant and we assume $n$ large enough (i.e. $n\ge n_0(r)$ for some $n_0(r)$).

Moreover by \eqref{eq99}, Fubini's equality and Tchebyshev's inequality, if we let
$$ A= \{ E \mid \nu_{n,M} \left(\{ F\mid d(E,F)<  C{n/\log(M/n)}  \}\right) \le 4 e^{-nM}\}$$  we have
$$\nu_{n,M}(\cl B_n \setminus A)\le1/2,$$
and hence
$$1/2 \le \nu_{n,M}( A)\le N(A, d, C{n/\log(M/n)}) 4 e^{-nM},$$
so that choosing $M=n\exp{n^\tau}$ ($0<\tau<1$) we find (assuming $n\ge n_0(\tau)$)
$$N(A,d,Cn^{1-\tau}) \ge (1/8) \exp{ (n^2\exp{n^\tau} )} \ge   \exp{  \exp{n^\tau} }. $$
A fortiori we have 
$$M(\cl B_n ,d,Cn^{1-\tau}) \ge M(A,d,Cn^{1-\tau}) \ge N(A,d,Cn^{1-\tau}) \ge \exp{  \exp{n^\tau} }. $$
\end{rem}
\begin{rem}\label{or} A natural question arises on the behaviour of the best possible constant $c_r$
in Theorem \ref{t1} when $r\to \infty$. In \cite{Sz5} Szarek proves that this is $O(1/r)$. More precisely, given $r > 2$ and $ n \in  \N$, his result is that every $n$-dimensional normed space embeds $r$-isomorphically  in $\ell_\infty^m$ for some $m\le n\exp{c'n/r}$.
Consequently (see Lemma \ref{l22} and note that we may assume $r\le \sqrt n$), the space ${\cl B}_n$ admits an $r$-net of cardinality at most
$\exp{\exp {c''n/r}}$
where $c',c''$ are positive absolute constants.
See also \cite[Cor. 1.3]{Bar} for a related result.
\end{rem}
\begin{rem}  We return to the set $S_n(X)$
of $n$-dimensional subspaces of $X$ defined in Lemma \ref{l22}. 
We first observe that, if  a  Banach space $X$  contains $\ell_\infty^n$'s uniformly,
then $S_n(X)=\cl B_n$, so the metric entropy of $S_n(X)$ is maximal.
In sharp contrast, if $X=\ell_2$, $S_n(X)$ is reduced to $\ell_2^n$ so
  the metric entropy of $S_n(X)$ is minimal.\\
  Assume that a  Banach space $X$ has the following property:
  any $n$-dimensional subspace $E\subset X$ embeds $c$-isomorphically
  into a fixed Banach space $X_n$ with $\dim(X_n)=m $.
  Consider $0<{\xi}<1/n$ and let $R=(1+{\xi} n)(1-{\xi} n)^{-1}$.
  Then Lemma \ref{l22}  implies
  $$N(S(n,X), c R)\le (1+2/{\xi})^{nm}.$$
  If this holds for $m=m(n)$ with $m(n)=2^{o(n)}$ then (choosing $\xi=\vp/n$)
  we have
   $$\log\log N(S_n(X),d, 1+\vp)\in o(n)$$
   for any $0<\vp<1$.\\
  When $X=\ell_1$ or $L_1$, it is proved in \cite{BLM} (see also \cite{S} and \cite{T}) that 
  for any $\vp>0$ 
  any $n$-dimensional subspace $E\subset X$ embeds $1+\vp$-isomorphically
  into $\ell_1^{c_\vp n (\log n)^3}$. This implies
  that $N(S_n(X),d, 1+\vp)$ is much smaller than
  the metric entropy  of $  \cl B_n$.
  In particular we have
   $$\log\log N(S_n(X),d, 1+\vp)\in o(n).$$
 By our initial observation, if,  for some $\vp>0$, we have
   $$\log\log N(S_n(X),d, 1+\vp)\in o(n)$$
   then $X$ does not contains $\ell_\infty^n$'s uniformly, and hence has finite cotype.\\
   It is tempting to wonder whether the converse holds (recall that 
   $X=\ell_1$ or $L_1$ are examples of cotype 2 spaces). This question is reminiscent
   of the conjecture formulated in \cite{Pm}.
\end{rem}
\n{\bf Problem.} Characterize the infinite dimensional Banach spaces $X$
such that  $$\log\log N(S_n(X),d, 1+\vp)\in o(n) \quad (n\to \infty)$$
for some (or all) $\vp>0$.

Note that this includes all $L_p$ spaces ($1\le p<\infty$) by \cite[Th. 7.3 and Th. 7.4]{BLM}.

Same question for  the class $SQ_n(X)$
of all $n$-dimensional subspaces of quotients of $X$. In that case, the $o(n)$-condition
implies that $X$ does not contains $\ell_1^n$'s uniformly, and hence has finite type.
\begin{rem}
Our main result suggests a number of other questions.
Among them in the style of \cite{TJ}
it is natural  to wonder how behaves the metric entropy of 
the class of spaces in $\cl B_n$ 
with either $1$-symmetric or $1$-unconditional bases.
\end{rem}

 \section{Large local metric entropy}\label{s2-}
In this section we formulate a refinement of our main result in \S \ref{s1}: even the balls 
around a single fixed space $E$ in $\cl B_n$
have extremely large metric entropy. 
While for {\it some} $E$ this can be deduced a priori from \S \ref{s1}  by an elementary   argument involving the maximal metric entropy of a ball of a fixed radius, it seems more surprising that it holds
for {\it all} $E$.  However, it turns out that, using \cite{QS}, we can essentially reduce to the case
  $E=\ell_2^n$.

More precisely we have:
 
\begin{thm}\label{local}
  For all $r,R$ such that  $1<r<R$ there is a constant  
  $c>0$ such that  for any $n$  assumed $n$ large enough (i.e. $n\ge n_0(r,R)$), for any $n$-dimensional Banach space $E$  
there is a subset $T_n\subset \cl B_n$ 
  with $$|T_n|\ge \exp\exp cn$$
  such that 
  $$\forall s, t\in T_n\quad d(t, s)> r.$$
  and
  $$\forall t\in T_n\quad d(t, E)<R.$$
\end{thm}
 \begin{rem} Note that it would be  optimal if $R$ could be any number such that
 $r<R^2,$
 with  $c=c_{r,R}$, but we could not reach this degree of precision.
 \end{rem}
 The proof is a simple modification of the proof of Theorem \ref{t1}  (lower bound).
\begin{lem} \label{lemloc}  Let $E$ be any $n$-dimensional space and $0<\theta<1$. Fix $c>0,C\ge 1 $.
Assume there a finite subset $ (x_t,x^*_t)_{t\in T} \subset E\times E^*$ with cardinality
  $|T|\ge \exp cn $
  such that $\forall s\not=t\in T$ we have
  $$\|x_t\|_E\le C \quad \|x_t\|_{E^*}\le C$$
  and 
  $$x^*_t(x_t)=1\quad |x^*_s(x_t)|\le \theta  .$$
  Then for any $r<1/\theta$, assuming $n$ large enough  (i.e. $n\ge n_0(r,\theta,c)$)
  there is a family of $n$-dimensional spaces $\{E_x\mid x\in {\cl X}\}$  with $|{\cl X}|\ge \exp {\exp (cn/2)} $ such that
  $$ \forall  x\in {\cl X}\quad d(E,E_x)\le C^2/\theta$$ and
  $$  \forall y\not=x\in {\cl X}\quad    d(E_y,E_x)\ge r . $$
Therefore, if $R=C^2/\theta$ and if $B_{E,R}$ denotes the $d$-ball centered in $E$ with radius $R$ in $\cl B_n$
we have 
  $$\exp {\exp (cn/2)} \le M(B_{E,R},d, r ) \le   N(B_{E,R},d, \sqrt r )  . $$
\end{lem}
\begin{proof} 
We assume $|T|$ even and we let $\cl A$ denote the set of all subsets   $x\subset \cl A$
with $|x|=|T|/2$. We then define $E_x$ as   equal to $E$  equipped with the norm
  $$\forall a\in E \quad \|a\|_{E_x}=  \max\{ C^{-1}\theta\|a\|_{E} \ , \  \sup_{t\in x} | \xi_t(a)  |\} .$$
Note that  
$$\forall a\in E \quad    C^{-1}\theta\|a\|_{E}  \le   \|a\|_{E_x}\le    C \|a\|_{E}  ,$$
and hence
$$d(E,E_x)\le C^2/\theta.$$
Moreover, for any $y\not=z\in \cl A$ we can find $s\in y\setminus z$  
so that 
$\|x_s\|_{E_y}=1$ but $\|x_s\|_{E_z}\le \theta$, and hence $\|E_z\to E_y\|\ge \theta^{-1}$.
Thus the proof can be completed just as above for Theorem \ref{t1} (lower bound).
\end{proof}

\begin{proof}[Proof of Theorem \ref{local}] We first observe that, e.g. by Lemma \ref{l11}, the assumption of Lemma
\ref{lemloc} is satisfied when $E=\ell_2^n$ for any $0<\theta<1$ (with $x_t=\xi_t$ and $C=1$). Moreover, by adjusting $c$   
the same still holds if $E$ admits either  a subspace (resp. a quotient)  or a subspace of a quotient  of proportional dimension (with a fixed proportion)
which is isometric to a Hilbert space. Indeed, the required  extension (resp. lifting)  is provided by the
classical Hahn-Banach Theorem.
Then the general case follows
from  Milman's proportional QS-theorem from \cite{QS} (see also \cite[p. 108]{P1}).
Indeed, given $r<R$, we may choose $0<\theta<1$ and $C>1$ so that
$r<\theta^{-1}$ and $C^2/\theta<R$. Then by \cite{QS} (see also \cite[p. 108]{P1}) there
is a proportion $0<\d<1$ (depending only on $C$) such that any $E\in \cl B_n$ admits a quotient of a subspace
of dimension $>\d n$ that is $C$-isomorphic to a Hilbert space. By the same (extension/lifting) argument as for the isometric case,
  the assumption (and hence the conclusion) of Lemma \ref{l11} holds for some adjusted value of $c$.
\end{proof}

\section{The matricial case}\label{s2}

We now turn to operator spaces, i.e. closed subspaces of the algebra $B(H)$ of all bounded operators on Hilbert space. The analogue of the norm becomes a sequence of norms: for each $N$ we consider the space $M_N(E)$ of 
$N\times N$ matrices with entries in $E$ equipped with the norm induced by $B(H\oplus\cdots\oplus H)$.
It is then customary to think of the sequence $\{\|. \|_{M_N(E)}\mid N\ge 1\}$ as the operator space analogue of the usual norm (corresponding to $N=1$). We refer to the books \cite{ER,P4}
for all unexplained terminology and for more background.

We wish to study the metric space $OS_n$ equipped with the
completely bounded  analogue
$d_{cb}$ of the distance $d$.
However,   
although it is complete, this space is not compact,
and by \cite{JP} it is not even separable.
Nevertheless, one may study the distance associated to a fixed size 
$N\ge 1$ for the matrix coefficients,
as follows.

Let $u:\ E \to F$ be a linear map between operator spaces.
We denote  
$$u_N =Id\otimes u :\ M_N(E)\to M_N(F).$$
If $E,F$ are two operator spaces that are isomorphic as Banach spaces, we set
\begin{equation}\label{eq77}d_N(E,F)=\inf \{ \|u_N\| \|(u^{-1})_N\| \}\end{equation}
where the inf runs over all the isomorphisms $u:\ E \to F$.\\
We set $d_N(E,F)=\infty$ if $E,F$ are not isomorphic.\\
Recall that $$\|u\|_{cb}=\sup\nolimits_{N\ge 1} \|u_N\|.$$
Recall also that, if $E,F$ are completely isomorphic, we set
$$d_{cb}(E,F)=\inf \{ \|u\|_{cb} \|u^{-1}\|_{cb} \}$$
where the inf runs over all the complete isomorphisms $u:\ E \to F$.\\
When $E,F$ are both $n$-dimensional, using a compactness argument,
  it is easy to  show that $d_N(E,F)=1$
iff there is an isomorphism $u:\ E\to F$ such that $u_N$ is isometric.
Moreover, again by a compactness argument, one can check easily that 
$$d_{cb}(E,F)=\sup\nolimits_{N\ge 1} d_N(E,F).$$

Our main result is :

\begin{thm} \label{t2} 
For any $r>1$,
there are   positive constants $b_r ,c_r $
such that,  for any $n$  assumed large enough
(more precisely $n\ge n_0(r)$), we have for all $N\ge 1$
$$ \exp(\exp b_r nN^2) \le N({\cl OS}_n,d_N, r)\le \exp(\exp c_r  nN^2).$$
\end{thm}

As observed in \cite[Lemma 2.11]{P9}, for any $E\in {\cl OS}_n$ and $\vp>0$
there is $F\subset \ell^k_\infty \otimes M_N$ such that $d_N(E,F)\le (1-\d)^{-1}$
provided $k\le (1+2/\d)^{nN^2}$. The argument for this is entirely analogous to the one above
for Lemma \ref{l22}. Using this, the upper bound in Theorem \ref{t2} can be proved just like
 the one for Theorem \ref{t1}. For the lower bound, as in \cite{P9} we crucially use
 quantum expanders  following on \cite{Ha} (see \cite{P9} 
 for more information and  references). 
 
 Let $\tau_N$ denote the normalized trace on $M_N$.
We will denote by $S_\vp=S_\vp(n,N)\subset U(N)^n$ the set of all   $n$-tuples
$u=(u_j)\in U(N)^n$ such that 
  $\forall x\in M_N$,   we have
$$\|\sum u_j (x -\tau_N(x)I) u_j^* \|_{L_2(\tau_N)}\le \vp n  \| x\|_{L_2(\tau_N)}.$$

We will use the following  result from \cite{P9}.

\begin{thm}\label{goal}  
For any $0<\d<1$ there is a 
  constant $ \beta_\d>0$   such that for each $0<\vp<1$   and
for all sufficiently large integer $n$ (i.e. $n\ge n_0$   with $n_0$ depending
on   $\vp$ and $\d$) and for all  $N\ge 1$,  there is a  subset $T\subset S_\vp(n,N)$ 
with $|T|\ge \exp{\beta_\d nN^2}$  that is $\delta$-separated in the following sense:
$$\forall s=(s_j) \not= t=(t_j) \in T\quad \|\sum s_j \otimes \ovl{ t_j}\|\le (1-\delta) n,  $$
where the norm is in the space $M_N(M_N)$ or equivalently $M_{N^2}$. 
\end{thm}
\begin{proof}[Proof of Theorem \ref{t2} (Lower bound)] 
Let $T$ be as in the preceding Theorem. Let $K =|T|\ge \exp{\beta_\d nN^2}$ (assumed even) and let
$\cl A$  denote again the class of subsets $x\subset T$ such that $|x|=K/2$, so that
again for all $x\not= y\in \cl A$ we have $x\not\subset y$, and 
$|\cl A|\ge \gamma 2^K/\sqrt{K}$. 
We denote by $e_{jj}$ the canonical basis of $\ell_\infty^n$.
We then define
for any subset $x\in \cl A$, (note $x\subset T \subset U(N)^n$)
$$u_j^x= e_{jj}  \oplus [ \oplus_{t\in x} t_j ] ,$$
and we define 
$$F_x={\rm span}[u_j^x,1\le j\le n]\subset  \ell_\infty^n\oplus_{\infty} [\ell_\infty(x)\otimes M_N]  .$$
Here the space $\ell_\infty(x)\otimes M_N$ is equipped with its (unique) $C^*$-norm,
namely the norm defined for a function $f:\ x\to M_N$ by $\|f\|=\sup\nolimits_{t\in x} \|f(t)\|_{M_N}$.
With this notation, $ [ \oplus_{t\in x} t_j ]\in \oplus_{t\in x}  M_N$ is identified with the function
$f_j:\ x\to M_N$ defined by $f_j(t)=t_j\ \forall t\in x$.\\
We equip $M_N\otimes M_N$ (resp. $M_N \otimes [\ell_\infty(x)\otimes M_N] $)
with the classical minimal (or spatial) tensor norm, or equivalently 
the norm of  $M_{N^2}$ (resp.  the sup-norm  on $x$ of     $M_{N^2}$-valued functions). 
Let ${\theta}=1-\delta$. Then, whenever $x\not= y\in \cl A$,  since $x\not\subset y$  there is $s\in x\setminus y$
and hence (assuming as we may $(1-\delta)n\ge 1$)
$$\|\sum \bar s_j \otimes u_j^x\|=n\quad \|\sum \bar s_j \otimes u_j^y\|\le (1-\delta)n=\theta n.$$
We will now use
   the metric entropy estimate
of the unit ball of $B(\ell_1^n, E_x)$, 
just like before. 
Fix $r$ such that  $1<r<{\theta} ^{-1}$ and let $\eta$ be as before so that $1+\eta = 1/(r{\theta} )$.  We find that
for a fixed $x\in \cl A$     we have 
 $$|\{y\in \cl A\mid d_{N}(F_x,F_y)<r\}|\le (1+4n/\eta)^{2n^2}.$$
   Let $\cl X\subset \cl A$ be a maximal subset such that
 $d_{N}(F_x,F_y)\ge r$ for any $x\not= y\in \cl X$.
 Then $\cl A$ is covered by the $d_N$-balls of radius $r$ centered in the points of $\cl X$, and hence by the claim
 $$ |\cl A|\le |\cl X| (1+4n/\eta)^{2n^2}\le |\cl X|  \exp{8n^3/\eta } .$$
 Thus we conclude as before
 $$|\cl X| \ge  |\cl A| \exp{-8n^3/\eta},$$
 and for $n$ large enough, the announced lower bound follows, just like for Theorem \ref{t1}.
 \end{proof}

 \begin{rem}
 Using  the stronger separation property considered in \cite[Lemma 2.14]{P9}, it is easy to check 
 the same lower bound
 (but possibly for a smaller value of  
   $b_r$) as in Theorem \ref{t2} for the class 
   denoted by ${\cl H\cl H}_n$ of $n$-dimensional Hilbertian homogeneous
 operator spaces.  However, this argument requires    $N^2/n\ge N_0$. Note that 
 if $E,F$ are two $n$-dimensional Hilbertian homogeneous
 operator spaces, then $d_{cb}(E,F)=\|Id:E\to F\||       \|Id:F\to E\|$. This is due
 to C. Zhang (see \cite[p. 217]{P4}). With this fact  the completion of the proof is much simpler
for this case.
\end{rem} 
 \begin{rem} At the time of this writing, we do not see how to prove the analogue  for $OS_n$ of Szarek's result
described in Remark \ref{or}. However we do have a very simple argument
to show that 
for any $r>1$ $\forall    n\ge (r\log r)/4 \ \forall N\ge 1$ we have $$ \log\log N({\cl H\cl H}_n,d_N,r)\le (c/r) nN^2$$ 
for some absolute constant $c>0$. The idea is that any $E\in {\cl H\cl H}_{2n}$ is completely $2$-isomorphic
to $F\oplus F$ where $F$ is any $n$ dimensional subspace of $E$ (note that, by homogeneity,  these are all mutually completely isometric).
Thus if we let $ m(n,N, r)$ be the smallest $k$ such that for any $E\in {\cl H\cl H}_{n}$  
there is $\hat E\subset \ell_\infty^k(M_N)$ with $d_N(E,\hat E)\le r$, then we have \begin{equation}\label{eq88}m(2n,N,2r)\le 2m(n,N,r).\end{equation} But 
as we already mentioned we also know that 
$m(n,N,r)\le \exp{ (2nN^2r(r-1)^{-1})}$  (see \cite[Lemma 2.11]{P9}), and in particular, say, $m(n,N,2)\le \exp{ (4nN^2)}$.
Iterating \eqref{eq88}  (assuming for simplicity that $n$ is a power of $2$) we obtain
$$m(n,N,2^{k+1})\le 2^k m(n/2^k,N,2)\le 2^k\exp{ (4nN^2 2^{-k} )}.$$
Thus if $r=2^k\le \exp{ (4n /r )}$ (i.e. if $n\ge (r\log r)/4$) then $m(n,N,2r)\le \exp{ (8nN^2 /r )}$ for any $N\ge 1$.
From this the announced result follows, using the operator space version of Lemma \ref{l22}.
 \end{rem}
\section{Why quantum expanders ?}\label{s3}
  It seems worthwhile to clarify why quantum expanders are needed in the proof of
   Theorem \ref{goal} and why the latter is   crucial to extend \S 1.
   
  A direct generalization of Lemma \ref{l11} to matrices would require to produce      a  subset $T\subset U(N)^n$ 
with $|T|\ge \exp{\beta nN^2}$  such that  
\begin{equation}\label{eq44}\forall u=(u_j) \not= v=(v_j) \in T\quad \|\sum u_j \otimes \ovl{ v_j}\|\le (1-\d) n, \end{equation}
with   $\beta,\delta$ as in Theorem \ref{goal}.  One could argue
similarly and consider a maximal subset $T\subset U(N)^n$  satisfying \eqref{eq44}.
Then denoting by $m$ the normalized  Haar measure on $U(N)^n$,
we note that, by maximality,  $U(N)^n$ is covered by the
union of the sets $$C_u(\d)=\{v\in U(N)^n\mid \|\sum u_j \otimes \ovl{ v_j}\|> (1-\delta) n\}$$
for $u\in T$. Thus if we set
$$F(\d)=\sup_{u\in T} m(C_u(\d))$$
we find
$$1\le |T| F(\d).$$
In \S 1 we use the easy fact that when $N=1$ there is $c>0$ 
so that we have $m(C_u(\d))\le \exp -cn$
for all $u\in \T^n$ or actually for all $u\in \{-1,1\}^n$. 
For arbitrary $N$, one can show by concentration of measure arguments an upper bound
of the form $m(C_u(\d))\le \exp -cnN$. This follows e.g. from a Gaussian matrix result
 due to Haagerup and Thorbjoernsen \cite[Th. 3.3]{HT2} together with 
 a comparison principle between Gaussian and unitary random matrices
(see \cite[p. 82]{MP}). 

However, in our present setting
the  corresponding result we would need to prove Theorem \ref{t2} by
  direct analogy with \S 1 should be \begin{equation}\sup_{u \in U(N)^n}\label{eq33}m(C_u(\d))\le \exp -cnN^2.\end{equation}
This  is simply not true: Take $u_j=I$ for all $j$, and note
$\| \sum u_j \otimes \ovl{ v_j} \|= \|\sum v_j\|\ge \Re(\sum \langle v_j e,e\rangle) $
for any fixed $e$ in the unit sphere of $ \ell_2^N$.
Then  $\cap_{1\le j\le n}\{ v \mid  \Re( \langle v_j e,e\rangle) > (1-\d)   \} \subset C_u(\d)$,
and hence (here   $m_1$ denotes Haar measure on $U(N)$)
$$m(C_u(\d) )\ge (m_1\{ v\in U(N) \mid \Re( \langle v e,e\rangle) > (1-\d) \})^n.$$
But it is well known and easy to check that for any fixed $\d>0$, there is $c_1>0$
such that $m_1\{ v\in U(N) \mid \Re( \langle v e,e\rangle) > (1-\d) \}\ge \exp -c_1 N$
for all $N$ large enough. Thus we must  have 
$m(C_u(\d))\ge \exp -c_1 nN$ when $N\to \infty$, contradicting  
the desired bound $m(C_u(\d))\le \exp -cnN^2$ at least if we want it to hold for any $u$.

However, let us assume instead that
  \begin{equation}\label{eq32}m\{u\in U(N)^n\mid m(C_u(\d))\le \exp -cnN^2\} \ge a>0,\end{equation}
and let $A=\{u\in U(N)^n\mid m(C_u(\d))\le \exp -cnN^2\}$.
Then let $T$ be  a maximal subset $T\subset A$ satisfying \eqref{eq44}.
Again by maximality we find $A\subset \cup_{u\in T} C_u(\d)$
and hence, since $m(C_u(\d))\le \exp -cnN^2$ for $u\in A$, we now have
$a\le m(A)\le |T|  \exp -cnN^2$ and we conclude
$$|T|\ge a  \exp cnN^2,$$
which implies Theorem \ref{goal}.

In \cite{P9} we show that this can be applied when $A$ is the set $S_\vp(n,N)$
for $0<\vp<1$ small enough with respect to $\d$. 
For the convenience of the reader we will outline the argument
of \cite{P9}, with special emphasis
on the key concentration of measure issues.\\
Assume $0<\d<1$. We define $0<\vp,\vp'<1$ by the equalities
$$3\vp'^{1/3}= (1-\d)/2\quad 2\vp= (1-\d)/2,$$
so that $$1-\d= 3\vp'^{1/3}+2\vp.$$
With these values (depending on $\d$) of $\vp,\vp'$, it is shown in \cite[Lemma 1.11]{P9} that if $u,v\in S_\vp(n,N)$ then
$$\inf_{U,V\in U(N)} N^{-1} \sum {\rm tr}  | Uu_jV-v_j|^2\ge 2n(1-\vp')\Rightarrow \|\sum u_j\otimes \bar v_j\|\le n(1-\d).$$
Equivalently  we have
 \begin{equation}\label{eq55}  C_u(\d)\subset \{ v\in U(N)^n \mid \sup_{U,V\in U(N)} \Re(N^{-1}  {\rm tr} ( \sum_1^n u_j Uv_j^*V^*))>\vp'n \}.\end{equation}
Let $f(U,V)=\Re(N^{-1}  {\rm tr} ( \sum_1^n u_j Uv_j^*V^*))$. Using the fact that
for any $U,U',V,V'$ in $U(N)$ we have
$$ |f(U,V)-f(U',V')|\le n (\|U-U'\|+\|V-V'\|),$$
together with the well known fact (see e.g. \cite[appendix]{Sz2}) that for any $\xi>0$  there is
an $\xi$-net  ${\cl N}_{\xi}\subset U(N)$ with respect to the operator norm
with  $|{\cl N}_{\xi}|\le (\gamma/{\xi})^{2  N^2}$ for some numerical constant $\gamma$,
we find that
 $$\sup_{U,V\in U(N)} f(U,V)\le \sup_{U,V\in {\cl N}_{\xi}}f(U,V)+2n\xi,$$
 and hence by \eqref{eq55} 
 $$m(C_u(\d)) \le \sum_{U,V\in {\cl N}_{\xi}} m\{ v\mid   \Re(N^{-1}  {\rm tr} ( \sum_1^n u_j Uv_j^*V^*))>\vp' n-2n\xi \} . $$
 We choose $\xi=\vp'/4$ so that $\vp'n-2n\xi = n\vp'/2$.
 By the translation invariance of $m$, $m\{ v\mid   \Re(N^{-1}  {\rm tr} ( \sum_1^n u_j Uv_j^*V^*))>\vp' n-2n\xi \}$
 is actually independent of $U,V$ and $u$ and equal
 simply to
 $m\{ v\mid   \Re(N^{-1}  {\rm tr} ( \sum_1^n v_j ))>\vp' n-2n\xi \}$.
 Thus we find
 $$m(C_u(\d)) \le |  {\cl N}_{\xi}|^2  m\{ v\mid   \Re(N^{-1}  {\rm tr} ( \sum_1^n v_j ))>n\vp'/2\} . $$
 But now the last measure can be efficiently estimated by a simple subGaussian argument
 since it is known (see e.g. \cite[\S 36.3]{HR}) that there is a numerical constant $c''>0$ such that 
 for any  $s>0$
     $$m\{\Re(N^{-1}  {\rm tr} ( \sum_1^n v_j ))> s \}\le \exp -\frac{c'' s^2N^2}{n}.$$
     Recalling the bound for $|{\cl N}_{\xi}|$ for $\xi=\vp'/4$  we obtain  
     $$m(C_u(\d)) \le (4\gamma/ \vp' )^{4N^2}  \exp{-\{c''(\vp'/2)^2N^2 n} \} .
     $$ 
     Note that for  any $0<\delta<1 $, when $n$ is large enough (more precisely $n\ge n(\d)$), there is
      $c=c_\d>0$ so that the last term is
     $\le \exp -c_{\d} nN^2$ for all $N\ge1$. Indeed,  recalling that $\vp'=\vp'_\d$, one can set e.g. $c_\d= (1/2)c''(\vp'/2)^2$
     and then determine $n(\d)$ by the requirement that $ (4\gamma/ \vp' )^{4}\le \exp c_{\d} n$
     for all $n\ge n(\d)$.
     Then  for all  $n\ge n(\d)$     we have for any $N\ge 1$  (recall here that $\vp=\vp_\d$)
     \begin{equation}\label{eq31} S_\vp(n,N)\subset \{u\in U(N)^n\mid m(C_u(\d)) \le \exp -c_{\d} nN^2\}.\end{equation}
 To complete this approach, we now need to check that for any $\vp>0$
 $m(S_\vp(n,N))$ is bounded below. This is proved in the appendix of \cite{P9},
 as an immediate consequence of the following inequality where $C$ is an absolute constant
 and $P$ is the orthogonal projection onto multiples of the $N\times N$-identity matrix
 $$\forall n,N\ge 1 \quad \int_{U(N)^n}  \|\sum_1^n  u_j \otimes \bar u_j (1-P) \| dm(u) \le C \sqrt{n}.$$
 Indeed, by Tchebyshev's inequality, this implies
 $$\forall n,N\ge 1 \quad m((S_\vp(n,N))\ge  1- (C/\vp) {n}^{-1/2}.$$
 Thus we obtain finally by \eqref{eq31} for any $0<\d<1$
 $$m\{u\in U(N)^n\mid m(C_u(\d))\le \exp -c_{\d}nN^2\} \ge 1- (C/\vp_\d) {n}^{-1/2},$$
 and if $n(\d)$ is chosen large enough this last term is $>1/2$ for all $n\ge n(\d)$.\\
Thus the answer to the question raised in the title of this section  is : Because  quantum expanders 
are the key to prove 
\eqref{eq32}  which can be substituted to the (false) \eqref{eq33} !

  \n\textbf{Acknowledgment.}  I am grateful to Oded Regev  for useful suggestions
  and references, to Mikael de la Salle for several  improvements and simplifications and to Olivier Gu\'edon
  for pointing out \cite{LRT}. Special thanks are due to S. Szarek for numerous  corrections, suggestions and
  improvements.

 \end{document}